\documentclass{amsart}
\usepackage[margin=1.2in]{geometry}
\usepackage{tikz}
\usepackage{cancel}
\usetikzlibrary{decorations.markings}
\newtheorem{theorem}{Theorem}[section]
\newtheorem{conject}[theorem]{Conjecture}
\newtheorem{lem}[theorem]{Lemma}

\theoremstyle{definition}

\theoremstyle{remark}
\theoremstyle{cor}
\newtheorem{cor}[theorem]{Corollary}
\newtheorem{remark}[theorem]{Remark}

\numberwithin{equation}{section}

\begin{document}

\title{On a partial sum related to the Euler totient function}


\author{Ankush Goswami}
\address{Department of Mathematics, University of Florida, 
Gainesville, Fl 32603}
\email{ankush04@ufl.edu}


%


\begin{abstract}
Recently, Bordell\'{e}s, Dai, Heyman, Pan and Shparlinski in \cite{Igor} considered a partial sum involving the Euler totient function and the integer parts $\lfloor x/n\rfloor$ function. Among other things, they obtained reasonably tight upper and lower bounds for their sum using the theory of exponent pairs and in particular, using a recently discovered Bourgain's exponent pair. Based on numerical evidences, they also pose a question on the asymptotic behaviour for their sum which we state here as a conjecture. The aim of this paper is to prove an asymptotic formula for the average of their sum in the interval $[1,x]$. We show via Perron's formula that this average is a certain weighted analogue of the sum that Bordell\'{e}s \textit{et al} considered in their paper. Further, we show that their conjecture is true under certain conditions. Our proof involves Perron's contour integral method besides some analytic estimates of the Riemann zeta function $\zeta(s)$ in the zero-free region.    
\end{abstract}
\maketitle
\section{Introduction}
Let $\varphi(n)$ denote the Euler's totient function which is defined as the number of positive integers less than and coprime to $n$. Then for any real $x\geq 2$, it is well-known that  
\begin{eqnarray}\label{ES}
\sum_{n\leq x}\varphi(n)=\dfrac{x^2}{2\zeta(2)}+O(x\log(x))
\end{eqnarray}
where as usual $\zeta(s)$ is the Riemann zeta function. Recently, Bordell\'{e}s, Dai, Heyman, Pan and Shparlinski (BDHPS) considered the following variant of the sum in (\ref{ES}):
\begin{eqnarray}\label{EvS}
S_{\varphi}(x):=\sum_{n\leq x}\varphi\left(\left\lfloor x/n \right\rfloor\right).
\end{eqnarray}
From (\ref{ES}) and the observation that 
\begin{eqnarray}
\sum_{n\leq x}\left\lfloor\dfrac{x}{n}\right\rfloor = \sum_{n\leq x}\tau(n)=x\log(x)+(2\gamma-1)x+O(\sqrt{x})
\end{eqnarray}
where $\gamma$ is the Euler's constant and $\tau(n)=\sum_{d|n}1$ is the divisor function, they proved the following result.
\begin{theorem}[BDHPS]\label{Shpar}
Uniformly for all $x\geq 3$,
\begin{eqnarray*}
\ \ \ \left(\dfrac{2629}{4009}\cdot\dfrac{1}{\zeta(2)}+\dfrac{1380}{4009}+o(1)\right)x\log(x)\geq S_{\varphi}(x)\geq \left(\dfrac{2629}{4009}\cdot\dfrac{1}{\zeta(2)}+o(1)\right)x\log(x).
\end{eqnarray*}
where $o(1)\rightarrow 0$ as $x\rightarrow\infty$. 
\end{theorem}
The upper and lower bounds in Theorem \ref{Shpar} were sharpened by J. Wu \cite{Wu} and very recently, by S. Chern \cite{Chern}. Based on numerical evidences that BDHPS obtain towards the end of their paper, they pose the following question which we state here as a conjecture.
\begin{conject}[BDHPS]\label{conj}
With $S_{\varphi}(x)$ defined in $(\ref{EvS})$, we have 
\begin{eqnarray*}
S_{\varphi}(x)=\dfrac{x\log(x)}{\zeta(2)}(1+o(1)),\;x\rightarrow\infty.
\end{eqnarray*}
\end{conject}
In fact, they cautiously believe that Conjecture \ref{conj} is true. The aim of this paper is to prove an averaged version of Conjecture \ref{conj} (see below). Before that, we state a few important remarks below, some of which are also highlighted in \cite{Igor}. 
\begin{remark}
If we define $\tau_x(n)$ by
\begin{eqnarray*}\label{xdiv}
\tau_{x}(n):=\sum_{\substack{d|n\\(d,\lfloor dx/n\rfloor)=1}}1,
\end{eqnarray*} 
then by interchanging sums, we can see that 
\begin{eqnarray*}\label{resdiv}
S_{\varphi}(x)=\sum_{n\leq x}\tau_x(n).
\end{eqnarray*}
In other words, let for each $n\geq 1$ we denote by $\mathcal{D}(x,n)=\left\{d|n: (d,\lfloor dx/n\rfloor)=1\right\}$. Then we have
\begin{eqnarray*}
\tau_x(n)=\sum_{d\in \mathcal{D}(x,n)}1=\left.\tau(n)\right\vert_{\mathcal{D}(x,n)}
\end{eqnarray*}
where by $\left.\tau(n)\right\vert_{\mathcal{D}(x,n)}$ we simply mean the restriction of the function $\tau(n)$ on the set $\mathcal{D}(x,n)$. Hence $S_\varphi(x)$ is the sum of the restricted divisor function $\tau_x(n)$ as described above. Unfortunately, unlike $\tau(n)$, the function $\tau_x(n)$ is not multiplicative and thus it is reasonably difficult to obtain an explicit asymptotic formula for $S_{\varphi}(x)$.  
\end{remark}
\begin{remark}
The proof of Theorem \emph{\ref{Shpar}} relies on the theory of exponent pairs and in particular, to a recently discovered exponent pair of Bourgain, combined with the so called $A$- and $B$- processes.

Our method of proof of the main results (see below) is different than the proof of Theorem \emph{\ref{Shpar}}. We use contour integral method and the theory of Dirichlet series. In particular, our proof uses analytic estimates of the Riemann zeta function $\zeta(s)$ in the zero-free region. 
\end{remark}
\begin{remark}
We note here that our consideration of the weighted sum (see below) involving the totient function makes some of our computations easier. Without the weight, it seems reasonably difficult to obtain an asymptotic estimate of $S_\varphi(x)$. The choice of our weight is standard as in the proof of the prime number theorem in \cite{Rada}. In fact, one can also choose a Riesz-type weight \emph{(see \cite{Mont})} and get similar asymptotic formula. 
\end{remark}
%
\section{Main results}
\begin{theorem}\label{main}
Uniformly for $x\geq 2$ we have
\begin{eqnarray}\label{maineq}
\int_{1}^xS_{\varphi}(t)\;dt=\dfrac{x^2\log x}{2\zeta(2)}+O(x^2).
\end{eqnarray}
\end{theorem}
\begin{cor}\label{cor1}
Assume that there exists an $\alpha\in\mathbb{R}$ such that
\begin{eqnarray*}
S_{\varphi}(x)=\alpha x\log x(1+o(1)),
\end{eqnarray*}
as $x\rightarrow\infty$. Then
\begin{eqnarray}
\alpha=\dfrac{1}{\zeta(2)}.
\end{eqnarray}
\end{cor}
\begin{cor}\label{cor2}
Suppose there exist functions $S^{*}_\varphi(x)$ and $E_\varphi(x)$ such that 
\begin{eqnarray}
S_\varphi(x)=S^{*}_\varphi(x)+E_\varphi(x), \;\;x\geq 2
\end{eqnarray}
where $E_\varphi(x)=o(x\log x)$ and $S^{*}_\varphi(x)$ is monotonically non-decreasing. Then 
\begin{eqnarray}
S_\varphi(x)=\dfrac{x\log x}{\zeta(2)}(1+o(1))
\end{eqnarray} 
as $x\rightarrow\infty$. In other words, we can formally differentiate both sides of \emph{(\ref{maineq})} to get the conjectured asymptotic formula for $S_\varphi(x)$.
\end{cor}
\begin{cor}\label{cor3}
Let for $1\leq h<x$ and 
\begin{eqnarray}
L:=\int_{x}^{x+h}(S_\varphi(t)-S_\varphi(x))\;dt
\end{eqnarray}
we have
\begin{eqnarray}\label{coreq1}
L\ll h^2\log x+x^2.
\end{eqnarray}
Then Conjecture \emph{\ref{conj}} is true.
\end{cor}
\begin{cor}\label{cor4}
Let $T^{a}_\varphi(x):=\displaystyle\sum_{n\leq x}\Phi(n)\left(1-\dfrac{n}{x}\right)$ where $\Phi(n):=\displaystyle\sum_{d|n}(\varphi(d)-\varphi(d-1))$. Also, let the coefficients $\Phi(n)$ satisfy the condition
\begin{eqnarray}\label{bou}
\Phi(n)\geq -C\log n,\;n\geq 1,
\end{eqnarray}
for some $C\geq 0$. Then Conjecture \emph{\ref{conj}} is true.
\end{cor}
\section{Some preliminaries}
We introduce some notations and conventions below. 
\subsection{Notations}
We denote by $\mathbb{N}, \mathbb{R}, \mathbb{C}$ for the set of natural numbers (including zero), real numbers and complex numbers respectively. The letter $n$ will always denote a positive integer. For a positive real number $x$, we denote by $[x]$ the integral part of $x$ which is the nearest integer $\leq x$. Unless otherwise stated, $s=\sigma+it$ will always denote a complex number with $\sigma=\text{Re}(s)$ and $t=\text{Im}(s)$. $\zeta(s)$ denotes the Riemann zeta function which is defined as
\begin{eqnarray}
\zeta(s):=\sum_{n=1}^\infty\dfrac{1}{n^s},\;\;\sigma>1.
\end{eqnarray}
The Mobius function $\mu(n)$ is defined as $(-1)^{\nu(n)}$ when $n$ is square-free and zero otherwise where $\nu(n)$ is the count of the number of distinct prime factors of $n$.
\subsection{Basic theory of Dirichlet series and Perron integration} 
Let $f:\mathbb{N}\rightarrow\mathbb{C}$ be an arithmetical function. Consider the sum 
\begin{eqnarray}\label{sumasymp}
S_{f}(x):=\sum_{n\leq x}f(n).
\end{eqnarray}
and consider the Dirchlet series generating function for the sequence $\{f(n): n\in\mathbb{N}\}$ which is defined by
\begin{eqnarray}
D_f(s):=\sum_{n=1}^\infty\dfrac{f(n)}{n^s}.
\end{eqnarray}  
Then we have (see \cite{Apo, Mont, Ten})
\begin{theorem}[Perron's formula]\label{Perroni}
Let $D_f(s)$ be absolutely convergent for $\sigma_0>\sigma_a$. Then we have 
\begin{eqnarray}
\dfrac{1}{2\pi i}\int_{\sigma_0-i\infty}^{\sigma_0+i\infty}D_f(s)\dfrac{x^s}{s}ds=\sideset{}{'}\sum_{n\leq x}f(n)
\end{eqnarray}
where $\sum'$ indicates that the last term in the sum above is to be counted with weight $1/2$.
\end{theorem}
For convenience, it is sometimes easier to deal with a weighted form of Theorem \ref{Perroni}. There are several ways in which one can choose the weights (see \cite{Mont}). We only describe one such case.
\begin{theorem}\label{Perroni1}
Let $k$ be a non-negative integer and consider the weighted sum
\begin{eqnarray}
S_f^{k}(x):=\sum_{n\leq x}f(n)\left(1-\dfrac{n}{x}\right)^k,
\end{eqnarray}
then 
\begin{eqnarray}
S_f^{k}(x)=\dfrac{k!}{2\pi i}\int_{\sigma_0-i\infty}^{\sigma_0+i\infty}D_f(s)\dfrac{x^s}{s(s+1)(s+2)\cdots(s+k)}ds.
\end{eqnarray}
\end{theorem}
\section{Some initial transformations and results}
In this section, we will make some transformations on the sum $S_{\varphi}(x)$ and express it as a contour integral as described in the previous section. To do that, we first rewrite $S_{\varphi}(x)$ as follows. 
\begin{eqnarray}\label{Sphi}
S_\varphi(x)=\sum_{n\leq x}\left(\left\lfloor\dfrac{x}{n}\right\rfloor-\left\lfloor\dfrac{x}{n+1}\right\rfloor\right)\varphi(n)=T_{\varphi}(x)+O(x)
\end{eqnarray}
where
\begin{eqnarray}
T_{\varphi}(x)=\sum_{n\leq x}\left\lfloor\dfrac{x}{n}\right\rfloor(\varphi(n)-\varphi(n-1))
\end{eqnarray}
and we define $\varphi(0)=1$. We shall now look at the sum $T_\varphi(x)$ more closely. It is clear that
\begin{eqnarray}\label{Tphi}
T_{\varphi}(x)=\sum_{n\leq x}\Phi(n)
\end{eqnarray}
where 
\begin{eqnarray}
\Phi(n)=\sum_{d|n}(\varphi(d)-\varphi(d-1)).
\end{eqnarray}
We now consider the following weighted sum:  
\begin{eqnarray}\label{weightsum}
T_\varphi^{a}(x)&=&\sum_{n\leq x}\Phi(n)\left(1-\dfrac{n}{x}\right).
\end{eqnarray}
We then have
\begin{eqnarray}\label{smooth}
T_\varphi^{a}(x)&=&\dfrac{1}{x}\sum_{n\leq x}\Phi(n)\int_{n}^xdt=\dfrac{1}{x}\int_1^xT_\varphi(t)\;dt.
\end{eqnarray}
Next, consider the Dirichlet series 
\begin{eqnarray*}
D_\varphi(s)=\sum_{n=1}^\infty\dfrac{\Phi(n)}{n^s}.
\end{eqnarray*} 
By Dirichlet convolution, it is clear that 
\begin{eqnarray}\label{Diri}
D_\varphi(s)=\zeta(s)\sum_{n=1}^\infty\dfrac{\varphi(n)-\varphi(n-1)}{n^s}.
\end{eqnarray}
We next rewrite the series $\mathcal{D}_\varphi(s)$ in (\ref{Diri}) as follows.
\begin{eqnarray}\label{Dphi}
D_\varphi(s)&=&\zeta(s)\sum_{n=1}^\infty\dfrac{\varphi(n)}{n^s}-\zeta(s)\sum_{n=1}^\infty\dfrac{\varphi(n-1)}{n^s}\notag\\
&=& \zeta(s)\cdot\dfrac{\zeta(s-1)}{\zeta(s)}-\zeta(s)\cdot G_\varphi(s)\notag\\
&=&\zeta(s-1)-\zeta(s)\cdot G_\varphi(s)
\end{eqnarray}
where 
\begin{eqnarray*}
G_\varphi(s)=\sum_{n=1}^\infty\dfrac{\varphi(n-1)}{n^s}
\end{eqnarray*}
First, note that for any $s\in\mathbb{C}$ with $\sigma>2$, $D_\varphi(s)$ is absolutely convergent. Thus for $k=1$, we see from Theorem \ref{Perroni1} that
\begin{eqnarray}\label{Perravg}
T_\varphi^{a}(x)=\dfrac{1}{2\pi i}\int_{\sigma-i\infty}^{\sigma+i\infty}D_{\varphi}(s)\dfrac{x^{s}}{s(s+1)}ds.
\end{eqnarray}
We next prove the following result.
\begin{theorem}\label{tm1}
Let $s\in\mathbb{C}$. Then $D_{\varphi}(s)$ is meromorphic for $\sigma>0$ with a single pole at $s=1$ of order $2$. The residue of the pole at $s=1$ is
\begin{eqnarray}
\dfrac{\zeta(2)-\zeta'(2)}{\zeta^2(2)}-1-\sum_{n=1}^\infty\dfrac{\varphi(n)}{n^2(n+1)}\approx -0.8343893\cdots.
\end{eqnarray}  
\end{theorem}
To prove Theorem \ref{tm1}, we need the following result (see \cite{Apo}).
\begin{lem}\label{Apo}
For $x\geq 2$, $\sigma>1$ and $s\neq 2$, we have
\begin{eqnarray}
\sum_{n\leq x}\dfrac{\varphi(n)}{n^s}=\dfrac{x^{2-s}}{2-s}\dfrac{1}{\zeta(2)}+\dfrac{\zeta(s-1)}{\zeta(s)}+O(x^{1-\sigma}\log(x)).
\end{eqnarray}
\end{lem}
\begin{remark}
The proof of Lemma \emph{\ref{Apo}} follows using partial summation and the fact that 
\begin{eqnarray*}
\sum_{n=1}^\infty \dfrac{\mu(n)}{n^s}=\dfrac{1}{\zeta(s)}, \;\sigma>1.
\end{eqnarray*}
\end{remark}
\begin{proof}[Proof of Theorem \emph{\ref{tm1}}]
We rewrite $G_\varphi(s)$ after a change of variable $n\rightarrow n+1$ as
\begin{eqnarray}\label{Stint1}
G_\varphi(s)=1+\sum_{n=1}^\infty\dfrac{\varphi(n)}{(n+1)^s}.
\end{eqnarray}
Next, apply Steiljes integration as follows:
\begin{eqnarray}\label{Stint2}
G_\varphi(s)-1=\int_{1^{-}}^\infty (1+1/u)^{-s}d(A_\varphi(u))
\end{eqnarray}
where 
\begin{eqnarray*}
A_\varphi(u)=\sum_{n\leq u}\dfrac{\varphi(n)}{n^s}.
\end{eqnarray*}
In (\ref{Stint2}), we apply integration by parts along with Lemma \ref{Apo} to obtain
\begin{eqnarray}\label{Stint3}
G_\varphi(s)-1&=&\left.A_{\varphi}(u)(1+1/u)^{-s}\right\rvert_{1^-}^\infty-s\int_{1}^\infty \dfrac{A_\varphi(u)\;du}{u^2(1+1/u)^{s+1}}\notag\\
&=&\sum_{n=1}^\infty \dfrac{\varphi(n)}{n^s}-s(I_1(s)+I_2(s)+I_3(s))
\end{eqnarray}
where 
\begin{eqnarray}\label{inte}
I_1(s)&=&\dfrac{1}{\zeta(2)(2-s)}\int_{1}^\infty \dfrac{du}{u^s(1+1/u)^{s+1}}\notag\\
I_2(s)&=&\dfrac{\zeta(s-1)}{\zeta(s)}\int_{1}^\infty \dfrac{du}{u^2(1+1/u)^{s+1}}\notag\\
I_3(s)&\ll &\int_{1}^\infty \dfrac{\log(u)\;du}{u^{s+1}(1+1/u)^{s+1}}
\end{eqnarray}
We shall now estimate each of the integrals $I_1(s), I_2(s)$ and $I_3(s)$ in (\ref{inte}). We start with $I_1(s)$. 
\begin{eqnarray}\label{inte11}
I_1(s)=\dfrac{1}{\zeta(2)(2-s)}\int_{1}^\infty \dfrac{u\;du}{(u+1)^{s+1}}.
\end{eqnarray}
Making the substitution $v=1+u$ in (\ref{inte11}) and after a simple calculation, we obtain
\begin{eqnarray}\label{inte1}
I_1(s)=\dfrac{(s+1)}{2^s\zeta(2)s(s-1)(2-s)}.
\end{eqnarray}
To evaluate $I_2(s)$, we make the substitution that $v=1+1/u$ to obtain
\begin{eqnarray}\label{inte2}
I_2(s)&=&\dfrac{\zeta(s-1)}{\zeta(s)}\int_1^2\dfrac{dv}{v^{s+1}}\notag\\
&=&\dfrac{\zeta(s-1)}{s\zeta(s)}(1-2^{-s}).
\end{eqnarray}
For $I_3(s)$, we use integration by parts to obtain
\begin{eqnarray}\label{inte3}
sI_3(s)\ll \int_{1}^\infty \dfrac{1}{u(u+1)^s}<\infty, \;\;\sigma>0.
\end{eqnarray}
Thus from (\ref{Dphi}), (\ref{Stint3}), (\ref{inte1}) and (\ref{inte2}), it follows that 
\begin{eqnarray}\label{Dphi1}
D_\varphi(s)&=&\dfrac{(s+1)\zeta(s)}{2^s\zeta(2)(s-1)(2-s)}-\zeta(s)+\zeta(s-1)(1-2^{-s})\notag\\&&+\;s\zeta(s)I_3(s).
\end{eqnarray}
It is clearly seen that the right-hand side in (\ref{Dphi1}) has a pole at $s=1$ of order $2$ due to the factor $\zeta(s)/(s-1)$. Owing to the appearance of $\zeta(s-1)$, it is also tempting to think that the right-hand side of (\ref{Dphi1}) might also have a pole at $s=2$ which is simple. However, if we compute the residue at $s=2$ from the first expression in the right hand side of (\ref{Dphi1}), we get a contribution of $-3/4$ and computing similarly for the third expression in the right-hand side in (\ref{Dphi1}) gives us a contribution of $3/4$, combining which gives us a residue of $0$ at $s=2$. Thus $s=2$ is not a pole for the function $D_{\varphi}(s)$. We thus separate out the terms containing the factor $1/(s-2)$ and group them together (using partial fraction decomposition) to rewrite (\ref{Dphi1}) as follows
\begin{eqnarray}\label{Dnewphi}
D_\varphi(s)&=&\dfrac{\zeta(s)}{2^{s-1}\zeta(2)(s-1)}-\zeta(s)+\zeta(s-1)(1-2^{-s})-\dfrac{3\zeta(s)}{2^s\zeta(2)(s-2)}\notag\\&&+\;s\zeta(s)I_3(s).
\end{eqnarray}  
Finally, we conisder the last expression in the right-hand side of (\ref{Dphi1}) and  it is clear that $s=1$ is a simple pole of it since $sI_3(s)$ is analytic for $\sigma>0$ in view of (\ref{inte3}). Thus $s\zeta(s)I_3(s)$ is meromorphic for $\sigma>0$ since $\zeta(s)$ is meromorphic for $\sigma>0$. 

To compute the residue of the pole at $s=1$ for $D_\varphi(s)$, we obtain a slightly different representation of $D_\varphi(s)$. This new representation will be more convenient to deal with. For this, we start with equation (\ref{Dphi}) and make the change of variable $n\rightarrow n+1$. Next, we extract the term $n=1$ from the sum to obtain 
\begin{eqnarray}\label{Dphi2}
D_\varphi(s)&=&\zeta(s-1)-\zeta(s)\sum_{n=0}^\infty\dfrac{\varphi(n)}{(n+1)^s}\notag\\
&=& \zeta(s-1)-\zeta(s)(1+2^{-s})-\zeta(s)\sum_{n=2}^\infty\dfrac{\varphi(n)}{n^s}(1+1/n)^{-s}.
\end{eqnarray}   
We now apply Newton's generalized binomial theorem on the right-hand sum above to get,
\begin{eqnarray}\label{Dpphi1}
&\displaystyle\sum_{n=2}^\infty\dfrac{\varphi(n)}{n^s}(1+1/n)^{-s}=\displaystyle\sum_{n=2}^\infty\displaystyle\sum_{k=0}^\infty(-1)^k\binom{s+k-1}{k}\dfrac{\varphi(n)}{n^{s+k}}&\notag\\
&=\displaystyle \sum_{n=2}^\infty\dfrac{\varphi(n)}{n^s}-s\displaystyle\sum_{n=2}^\infty\dfrac{\varphi(n)}{n^{s+1}}+\displaystyle\sum_{n=2}^\infty\displaystyle\sum_{k=2}^\infty(-1)^k\binom{s+k-1}{k}\dfrac{\varphi(n)}{n^{s+k}}&\notag\\
&=\dfrac{\zeta(s-1)}{\zeta(s)}-1-\dfrac{s\zeta(s)}{\zeta(s+1)}+s+
\displaystyle\sum_{n=2}^\infty\displaystyle\sum_{k=2}^\infty(-1)^k\binom{s+k-1}{k}\dfrac{\varphi(n)}{n^{s+k}}.&
\end{eqnarray}
Combining (\ref{Dphi2}) and (\ref{Dpphi1}) we obtain
\begin{eqnarray}\label{Danophi1}
\ \ \ \ D_{\varphi}(s)=\dfrac{s\zeta^2(s)}{\zeta(s+1)}-\zeta(s)(s+2^{-s})-\zeta(s)\displaystyle\sum_{n=2}^\infty\displaystyle\sum_{k=2}^\infty(-1)^k\binom{s+k-1}{k}\dfrac{\varphi(n)}{n^{s+k}}.
\end{eqnarray}
Comparing (\ref{Dphi1}) and (\ref{Danophi1}), it can be seen readily that the third expression in the right-hand side of (\ref{Danophi1}) is
\begin{eqnarray}\label{moretest}
&\zeta(s)\displaystyle\sum_{n=1}^\infty\displaystyle\sum_{k=2}^\infty(-1)^k\binom{s+k-1}{k}\dfrac{\varphi(n)}{n^{s+k}}=\dfrac{s\zeta^2(s)}{\zeta(s+1)}-\dfrac{(s+1)\zeta(s)}{2^s\zeta(2)(s-1)(2-s)}&\notag\\&-\zeta(s)(s-1+2^{-s})-\;\zeta(s-1)(1-2^{-s})-s\zeta(s)I_3(s),&
\end{eqnarray}
which is meromorphic for $\sigma>0$ with a simple pole at $s=1$. Again, one checks easily that the right-hand side of (\ref{moretest}) does not have any pole at $s=2$. 
We next use (\ref{Danophi1}) to obtain the residue of the double pole of $D_\varphi(s)$ at $s=1$. To this end, we recall the Laurent series expansion of $\zeta(s)$ near $s=1$:
\begin{eqnarray}\label{Lzeta}
\zeta(s)=\dfrac{1}{s-1}+h(s)
\end{eqnarray} 
where $h(s)$ is analytic for $\sigma>0$. Combining (\ref{Danophi1}) and (\ref{Lzeta}), we rewrite $D_\varphi(s)$ as
\begin{eqnarray}\label{Dphirep}
&D_\varphi(s)=\dfrac{s}{\zeta(s+1)(s-1)^2}-\dfrac{s+2^{-s}}{s-1}-\dfrac{1}{s-1}\displaystyle\sum_{n=2}^\infty\displaystyle\sum_{k=2}^\infty(-1)^k\binom{s+k-1}{k}\dfrac{\varphi(n)}{n^{s+k}}\notag &\\&+\;L(s)&
\end{eqnarray}
where $L(s)$ is an analytic function for $\sigma>0$. Hence from (\ref{Dphirep}), it follows that
\begin{eqnarray}
\text{Res}_{s=1}D_\varphi(s)&=&\lim_{s\rightarrow 1}\dfrac{d}{ds}((s-1)^2D_\varphi(s))\notag\\
&=&\lim_{s\rightarrow 1}\dfrac{d}{ds}\left(\dfrac{s}{\zeta(s+1)}\right)-\dfrac{3}{2}-\displaystyle\sum_{n=2}^\infty\displaystyle\sum_{k=2}^\infty(-1)^k\binom{k}{k}\dfrac{\varphi(n)}{n^{k+1}}\notag\\
&=&\dfrac{\zeta(2)-\zeta'(2)}{\zeta^2(2)}-\dfrac{3}{2}-\displaystyle\sum_{n=2}^\infty\displaystyle\sum_{k=2}^\infty(-1)^k\dfrac{\varphi(n)}{n^{k+1}}\notag\\
&=&\dfrac{\zeta(2)-\zeta'(2)}{\zeta^2(2)}-\dfrac{3}{2}-\sum_{n=2}^\infty\varphi(n)\sum_{k=2}^\infty\dfrac{(-1)^k}{n^{k+1}}\notag\\
&=&\dfrac{\zeta(2)-\zeta'(2)}{\zeta^2(2)}-\dfrac{3}{2}-\sum_{n=2}^\infty\varphi(n)\left(\dfrac{1}{n+1}-\dfrac{1}{n}+\dfrac{1}{n^2}\right)\notag\\
&=&\dfrac{\zeta(2)-\zeta'(2)}{\zeta^2(2)}-1-\sum_{n=1}^\infty\dfrac{\varphi(n)}{n^2(n+1)}\approx -0.8343893\cdots.
\end{eqnarray}
\end{proof}\hfill  
\section{Contour integral method}
In this section, we use Cauchy's residue theorem to evaluate the contour integral $T_\varphi^{a}(x)$ in (\ref{Perravg}). This requires a closed contour and for this, we will first truncate the contour $\mathcal{C}$ at a suitable height. We will then make suitable deformation of this truncated contour to the left of the line $\sigma=1$ and make it a rectangular contour. We will need to estimate carefully the errors obtained in the process of trunctations and deformations. In estimating these errors, we require some well-known estimates for $\zeta(s)$ and we state them below (see \cite{Mont}).
\begin{theorem}\label{zerofree}
There is an absolute constant $c>0$ such that $\zeta(s)\neq 0$ for $\sigma>1-c/\log(|t|+4)$.
\end{theorem}
This is the classical zero-free region for $\zeta(s)$. Next we have
\begin{theorem}\label{zetaest}
Let $c$ be the constant in Theorem \emph{\ref{zerofree}}. If $\sigma>1-c/(2\log(|t|+4))$ and $|t|>7/8$, then
\begin{eqnarray}
|\log\zeta(s)|\ll \log\log(|t|+4)+O(1),\notag\\
\dfrac{1}{\zeta(s)}\ll \log(|t|+4).
\end{eqnarray}
On the other hand, if $1-c/(2\log(|t|+4))<\sigma\leq 2$ and $|t|\leq 7/8,$ then $\log(\zeta(s)(s-1))\ll 1$ and $1/\zeta(s)\ll |s-1|$.
\end{theorem}

Consider
\begin{eqnarray}\label{Perr1}
T_\varphi^{a}(x)=\dfrac{1}{2\pi i}\int_{\mathcal{C}}D_\varphi(s)\dfrac{x^s}{s(s+1)}ds.
\end{eqnarray}
Using (\ref{Dnewphi}) in (\ref{Perr1}), we obtain
\begin{eqnarray}\label{Tphi}
T_\varphi^{a}(x)=J_1(x)-J_2(x)+J_3(x)+J_4(x)
\end{eqnarray}
where 
\begin{eqnarray}\label{someint}
J_1(x)&=&\dfrac{1}{2\pi i}\int_{\mathcal{C}}\dfrac{\zeta(s)}{2^{s-1}\zeta(2)(s-1)}\dfrac{x^s}{s(s+1)}ds,\notag\\
J_2(x)&=&\dfrac{1}{2\pi i}\int_{\mathcal{C}}\zeta(s)\dfrac{x^s}{s(s+1)}ds,\notag\\
J_3(x)&=&\dfrac{1}{2\pi i}\int_{\mathcal{C}}\left(\zeta(s-1)(1-2^{-s})-\dfrac{3\zeta(s)}{2^s\zeta(2)(s-2)}\right)\dfrac{x^s}{s(s+1)}ds,\notag\\
J_4(x)&=&\dfrac{1}{2\pi i}\int_{\mathcal{C}}s\zeta(s)I_3(s)\dfrac{x^s}{s(s+1)}ds.
\end{eqnarray}
It only remains to estimate the integrals $J_1(x), J_2(x), J_3(x),$ and $J_4(x)$ and this involves a careful analysis of the integrands involved. We do this next.
\subsection{Estimation of $J_1(x), J_2(x), J_4(x)$}
In this section, we estimate $J_1(x), J_2(x)$ and $J_4(x)$. Our first task is to show that we can shift the line $\mathcal{C}$ to the left of the point $s=2$. Let $1<\sigma'<2$. Consider the rectangle $\mathcal{R}_1$ with vertices $\sigma-iT', \sigma+iT', \sigma'+iT'$ and $\sigma'-iT'$ as shown below. Since the integrands in $J_1(x), J_2(x)$ and $J_4(x)$ are analytic inside and on the rectangle $\mathcal{R}_1$, these integrals are zero around $\mathcal{R}_1$. We now show that the the contribution of each of the integrals $J_1(x), J_2(x)$ and $J_4(x)$ on the horizontal segments tend to zero as $T'\rightarrow\infty$.\newline
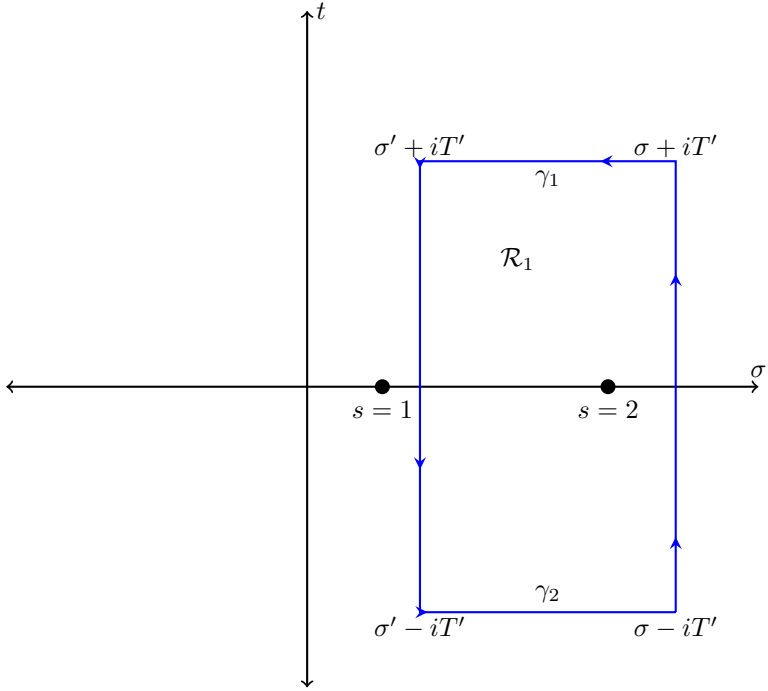
   \begin{tikzpicture}
[decoration={markings,
    mark=at position 1cm   with {\arrow[line width=1pt]{stealth}},
    mark=at position 4.5cm with {\arrow[line width=1pt]{stealth}},
    mark=at position 7cm   with {\arrow[line width=1pt]{stealth}},
    mark=at position 9.5cm with {\arrow[line width=1pt]{stealth}},
    mark=at position 13.5cm with {\arrow[line width=1pt]{stealth}},
    mark=at position 15.5cm with {\arrow[line width=1pt]{stealth}}
  }]
\draw[thick, <->] (-4,0) -- (6,0) coordinate (xaxis);

  \draw[thick, <->] (0,-4) -- (0,5) coordinate (yaxis);

  \node[above] at (xaxis) {$\sigma$};

  \node[right]  at (yaxis) {$t$};
\path[draw,blue, line width=0.8pt, postaction=decorate] (4.9,-3)
    -- node[midway, above right, black] {} (4.9,3)
    -- node[midway, below, black]{$\gamma_1$} (1.5,3)
    -- node[midway, above right, black] {} (1.5,-3) 
    -- node[midway, above, black] {$\gamma_2$}(4.9,-3);
\draw (1,0) node [circle, inner sep=2pt, fill]{};
\draw (1,-0.3) node{$s=1$};
\draw (4,0) node [circle, inner sep=2pt, fill]{};
\draw (4,-0.3) node{$s=2$};
\draw (4.9,3.2) node{$\tiny{\sigma+iT'}$};
\draw (1.5,3.2) node{$\tiny{\sigma'+iT'}$};
\draw (1.5,-3.2) node{$\tiny{\sigma'-iT'}$};
\draw (4.9,-3.2) node{$\tiny{\sigma-iT'}$};
\draw (2.8,1.7) node{$\mathcal{R}_1$};
\end{tikzpicture}\newline
From Theorem \ref{zetaest} we have
\begin{eqnarray}\label{Hor}
\dfrac{1}{2\pi i}\int_{\gamma_1\cup\gamma_2}\dfrac{\zeta(s)}{2^{s-1}\zeta(2)(s-1)}\dfrac{x^s}{s(s+1)}ds&\ll &\dfrac{x^{\sigma}\log T'}{T'^3}(\sigma-\sigma')\notag\\
\dfrac{1}{2\pi i}\int_{\gamma_1\cup\gamma_2}\zeta(s)\dfrac{x^s}{s(s+1)}ds&\ll &\dfrac{x^\sigma\log T'}{T'^2}(\sigma-\sigma')\notag\\
\dfrac{1}{2\pi i}\int_{\gamma_1\cup\gamma_2}s\zeta(s)I_3(s)\dfrac{x^s}{s(s+1)}ds&\ll &\dfrac{x^\sigma\log T'}{T'^2}(\sigma-\sigma').
\end{eqnarray}
Therefore the integrals $J_1(x), J_2(x)$ and $J_4(x)$ along horizontal segments $\gamma_1\cup\gamma_2$ tend to zero as $T'\rightarrow\infty$ and we have
\begin{eqnarray}
J_1(x)&=&\dfrac{1}{2\pi i}\int_{\mathcal{C}'}\dfrac{\zeta(s)}{2^{s-1}\zeta(2)(s-1)}\dfrac{x^s}{s(s+1)}ds\notag\\
J_2(x)&=&\dfrac{1}{2\pi i}\int_{\mathcal{C}'}\zeta(s)\dfrac{x^s}{s(s+1)}ds\notag\\
J_4(x)&=&\dfrac{1}{2\pi i}\int_{\mathcal{C}'}s\zeta(s)I_3(s)\dfrac{x^s}{s(s+1)}ds
\end{eqnarray}
where $\mathcal{C}'=[\sigma'-i\infty, \sigma'+i\infty]$. 

Next, we truncate the contour $\mathcal{C}'$ at a height $1\leq T\leq x$ (to be chosen later). Let $\mathcal{C}'_T=[\sigma'\pm iT, \sigma'\pm i\infty]$. We choose 
\begin{eqnarray}\label{parre}
\sigma'=1+1/\log x\notag\\ 
\sigma_1=1-c/\log T
\end{eqnarray}
where $c$ is a small positive constant. Next, deform the contour $C_1'=[\sigma'-iT,\sigma'+iT]$ to a rectangular contour $\mathcal{R}'$ that consists of line segments $C_1', C_2', C_3'$ and $C_4'$ joining the points $\sigma'-iT, \sigma'+iT, \sigma_1+iT$ and $\sigma_1-iT$ as shown below.\newline
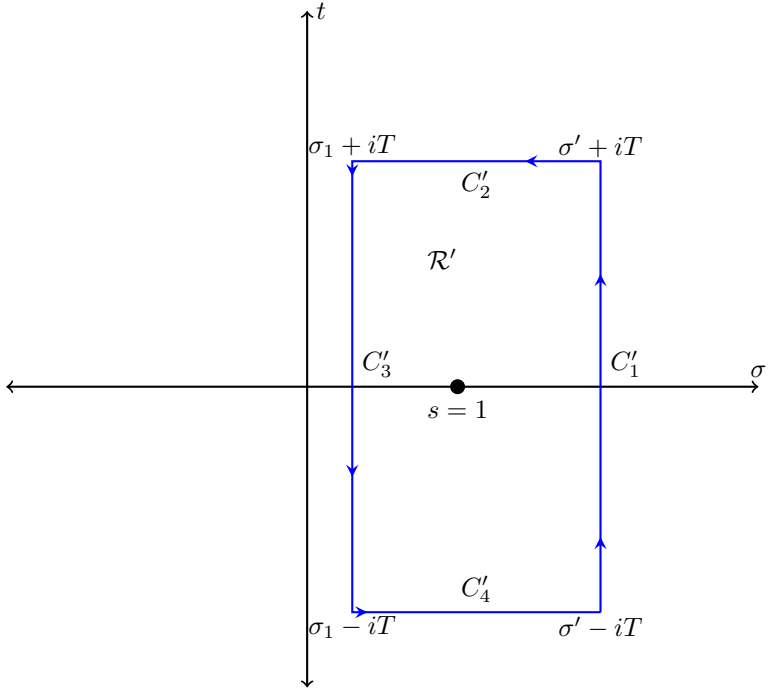
\begin{tikzpicture}
[decoration={markings,
    mark=at position 1cm   with {\arrow[line width=1pt]{stealth}},
    mark=at position 4.5cm with {\arrow[line width=1pt]{stealth}},
    mark=at position 7cm   with {\arrow[line width=1pt]{stealth}},
    mark=at position 9.5cm with {\arrow[line width=1pt]{stealth}},
    mark=at position 13.5cm with {\arrow[line width=1pt]{stealth}},
    mark=at position 15.5cm with {\arrow[line width=1pt]{stealth}}
  }]
\draw[thick, <->] (-4,0) -- (6,0) coordinate (xaxis);

  \draw[thick, <->] (0,-4) -- (0,5) coordinate (yaxis);

  \node[above] at (xaxis) {$\sigma$};

  \node[right]  at (yaxis) {$t$};
\path[draw,blue, line width=0.8pt, postaction=decorate] (3.9,-3)
    -- node[midway, above right, black] {$C_1'$} (3.9,3)
    -- node[midway, below, black] {$C_2'$}(0.6,3)
    -- node[midway, above right, black] {$C_3'$} (0.6,-3) 
    -- node[midway, above, black] {$C_4'$}(3.9,-3);
\draw (2,0) node [circle, inner sep=2pt, fill]{};
\draw (2,-0.3) node{$s=1$};
\draw (3.9,3.2) node{$\tiny{\sigma'+iT}$};
\draw (0.6,3.2) node{$\tiny{\sigma_1+iT}$};
\draw (0.6,-3.2) node{$\tiny{\sigma_1-iT}$};
\draw (3.9,-3.2) node{$\tiny{\sigma'-iT}$};
\draw (1.8,1.7) node{$\mathcal{R}'$};
\end{tikzpicture}\newline
Since $s=1$ is a double pole of the integrand in $J_1(x)$, using Cauchy's residue theorem we get
\begin{eqnarray}\label{res1}
\dfrac{1}{2\pi i}\int_{\mathcal{R}'}\dfrac{\zeta(s)}{2^{s-1}\zeta(2)(s-1)}\dfrac{x^s}{s(s+1)}ds=\left.\dfrac{d}{ds}\left(\dfrac{(s-1)^2\zeta(s)x^s}{2^{s-1}\zeta(2)(s-1)s(s+1)}\right)\right\rvert_{s=1}.
\end{eqnarray}
At this point we use the Laurent series expansion (\ref{Lzeta}) for $\zeta(s)$ near $s=1$ and the Taylor series expansion for $x^s/(2^{s-1}\zeta(2)s(s+1))$ near $s=1$ to obtain
\begin{eqnarray}\label{res2}
(s-1)\zeta(s)&=&1+(s-1)h(s)\notag\\
\dfrac{x^s}{2^{s-1}\zeta(2)s(s+1)}&=&\dfrac{x}{2\zeta(2)}+(s-1)\left(\dfrac{x\log x}{2\zeta(2)}-\dfrac{x(3+2\log 2)}{4\zeta(2)}\right)+O_x(|s-1|^2).
\end{eqnarray}
Thus from (\ref{res1}) and (\ref{res2}) we get
\begin{eqnarray}\label{resf}
\dfrac{1}{2\pi i}\int_{\mathcal{R}'}\dfrac{\zeta(s)}{2^{s-1}\zeta(2)(s-1)}\dfrac{x^s}{s(s+1)}ds=\dfrac{x\log x}{2\zeta(2)}+\kappa x
\end{eqnarray}
where 
\begin{eqnarray}
\kappa=\dfrac{h(1)}{2\zeta(2)}-\dfrac{(3+2\log 2)}{4\zeta(2)}.
\end{eqnarray}
For $J_2(x)$ and $J_4(x)$, we note that $s=1$ is a simple pole for each of the integrands. Thus a simple application of Cauchy's residue theorem yields
\begin{eqnarray}\label{resfd1}
\dfrac{1}{2\pi i}\int_{\mathcal{R}'}\zeta(s)\dfrac{x^s}{s(s+1)}ds&=&\dfrac{x}{2}\notag\\
\dfrac{1}{2\pi i}\int_{\mathcal{R}'}s\zeta(s)I_3(s)\dfrac{x^s}{s(s+1)}ds&=&\dfrac{xI_3(1)}{2}.
\end{eqnarray}
We next estimate the integrals $J_1(x), J_2(x)$ and $J_4(x)$ in each of the segments $C_T', C_2', C_3'$ and $C_4'$. 
\subsubsection{\underline{Estimation of $J_1(x)$}}
Using Theorem \ref{zetaest}, we easily see that
\begin{eqnarray}\label{e2}
J_1(x)-\dfrac{1}{2\pi i}\int_{\sigma'-iT}^{\sigma'+iT}\dfrac{\zeta(s)}{2^{s-1}\zeta(2)(s-1)}\dfrac{x^s}{s(s+1)}\ll \dfrac{x\log T}{T^2}.
\end{eqnarray}
and the contribution of the integral $J_1(x)$ on the contours $C_2', C_3'$ and $C_4'$ can be estimated again using Theorem \ref{zetaest} and we get
\begin{eqnarray}\label{e7}
\dfrac{1}{2\pi i}\int_{C_2'\cup C_4'}\dfrac{\zeta(s)}{2^{s-1}\zeta(2)(s-1)}\dfrac{x^s}{s(s+1)}ds&\ll &\dfrac{x^{\sigma'}\log T}{T^3}(\sigma'-\sigma_1)\ll \dfrac{x}{T^3}\notag\\
\dfrac{1}{2\pi i}\int_{C_3'}\dfrac{\zeta(s)}{2^{s-1}\zeta(2)(s-1)}\dfrac{x^s}{s(s+1)}ds &\ll &
x^{\sigma_1}(\log T)\int_{-T}^T\dfrac{dt}{(1+|t|)^3}\notag\\&&+\;x^{\sigma_1}\int_{-1}^1\dfrac{dt}{|\sigma_1+it-1|^2}\notag\\
&\ll &\dfrac{x^{\sigma_1}\log T}{T^2}+\dfrac{x^{\sigma_1}}{(1-\sigma_1)^2}\notag\\
&\ll &x^{\sigma_1}(\log T)^2. 
\end{eqnarray}
\subsubsection{\underline{Estimation of $J_2(x)$}}
We again use Theorem \ref{zetaest} to get
\begin{eqnarray}\label{e4}
&&J_2(x)-\dfrac{1}{2\pi i}\int_{\sigma'-iT}^{\sigma'+iT}\zeta(s)\dfrac{x^s}{s(s+1)}ds\ll \dfrac{x\log T}{T}.
\end{eqnarray}
and 
\begin{eqnarray}\label{e5}
\dfrac{1}{2\pi i}\int_{C_2'\cup C_4'}\zeta(s)\dfrac{x^s}{s(s+1)}ds&\ll& \dfrac{\log T}{T^2}x^{\sigma'}(\sigma-\sigma_1)\ll\dfrac{x}{T^2}\notag\\
\dfrac{1}{2\pi i}\int_{C_3'}\zeta(s)\dfrac{x^s}{s(s+1)}ds&\ll & x^{\sigma_1}(\log T)\int_{-T}^T\dfrac{dt}{(1+|t|)^2}+x^{\sigma_1}\int_{-1}^1\dfrac{dt}{|\sigma_1+it-1|}\notag\\
&\ll &\dfrac{x^{\sigma_1}\log T}{T}+\dfrac{x^{\sigma_1}}{1-\sigma_1}\notag\\
&\ll& x^{\sigma_1}\log T.
\end{eqnarray}
\subsubsection{\underline{Estimation of $J_4(x)$}}
Using Theorem \ref{zetaest} we get
\begin{eqnarray}\label{e6p}
J_4(x)-\dfrac{1}{2\pi i}\int_{\sigma-iT}^{\sigma+iT}s\zeta(s)I_3(s)\dfrac{x^s}{s(s+1)}ds\ll \dfrac{x\log T}{T}.
\end{eqnarray}
and
\begin{eqnarray}\label{e6}
\dfrac{1}{2\pi i}\int_{C_2'\cup C_4'}s\zeta(s)I_3(s)\dfrac{x^s}{s(s+1)}ds&\ll& \dfrac{\log T}{T^2}x(\sigma-\sigma_1)\ll\dfrac{x}{T^2}\notag\\
\dfrac{1}{2\pi i}\int_{C_3'}s\zeta(s)I_3(s)\dfrac{x^s}{s(s+1)}ds&\ll & x^{\sigma_1}(\log T)\int_{-T}^T\dfrac{dt}{(1+|t|)^2}+x^{\sigma_1}\int_{-1}^1\dfrac{dt}{|\sigma_1+it-1|}\notag\\
&\ll &\dfrac{x^{\sigma_1}\log T}{T}+\dfrac{x^{\sigma_1}}{1-\sigma_1}\notag\\
&\ll & x^{\sigma_1}\log T. 
\end{eqnarray}
\subsection{Estimation of $J_3(x)$}
The estimation of $J_3(x)$ requires a more careful analysis. The presence of $\zeta(s-1)$ in the integrand in $J_3(x)$ poses difficulties deforming the contour to the left of the line $\sigma=2$. To avoid this process, we rewrite $J_3(x)$  as 
\begin{eqnarray}\label{Jfin}
J_3(x)=J_{3,1}(x)-J_{3,2}(x),
\end{eqnarray}
where
\begin{eqnarray}\label{twointj}
J_{3,1}(x)&=&\dfrac{1}{2\pi i}\int_{\mathcal{C}}\zeta(s-1)(1-2^{-s})\dfrac{x^s}{s(s+1)}ds,\notag\\
J_{3,2}(x)&=&\dfrac{1}{2\pi i}\int_{\mathcal{C}}\dfrac{3\zeta(s)}{2^s\zeta(2)(s-2)}\dfrac{x^s}{s(s+1)}ds.
\end{eqnarray} 
Now we write $J_{3,1}(x)$ as 
\begin{eqnarray}\label{Jest}
J_{3,1}(x)&=&\dfrac{1}{2\pi i}\int_{\mathcal{C}}\zeta(s-1)\dfrac{x^s}{s(s+1)}ds-\dfrac{1}{2\pi i}\int_{\mathcal{C}}\zeta(s-1)\dfrac{(x/2)^s}{s(s+1)}ds.
\end{eqnarray}
Let 
\begin{eqnarray}\label{Jest0}
B(t):=\sum_{n\leq t}n=\dfrac{t^2}{2}+O(t).
\end{eqnarray}
Then it follows from Theorem \ref{Perroni1} and (\ref{Jest}), by choosing $k=1$, that 
\begin{eqnarray}\label{Jest1}
J_{3,1}(x)&=&\sum_{n\leq x}n\left(1-\dfrac{n}{x}\right)-\sum_{n\leq x/2}n\left(1-\dfrac{n}{x/2}\right)\notag\\
&=&\dfrac{1}{x}\int_{1}^xB(t)\;dt-\dfrac{2}{x}\int_{1}^{x/2}B(t)\;dt
\end{eqnarray}
Thus (\ref{Jest0}) and (\ref{Jest1}) yield
\begin{eqnarray}\label{Jest2}
J_{3,1}(x)&=&\dfrac{x^2}{6}-\dfrac{x^2}{24}+O(x)\notag\\
&=& \dfrac{x^2}{8}+O(x).
\end{eqnarray}
Next, we turn to $J_{3,2}(x)$. We first rewrite the integrand in $J_{3,2}(x)$ as
\begin{eqnarray}
\dfrac{3\zeta(s)}{2^s\zeta(2)(s-2)}&=&\dfrac{3(s-1)\zeta(s)}{2^s\zeta(2)}\dfrac{1}{(s-1)(s-2)}\notag\\
&=& \dfrac{3(s-1)\zeta(s)}{2^s\zeta(2)}\left(\dfrac{1}{s-2}-\dfrac{1}{s-1}\right)\notag\\
&=&F_2(s)-F_1(s)
\end{eqnarray}
where
\begin{eqnarray*}
F_1(s)&=& \dfrac{3\zeta(s)}{2^s\zeta(2)}\notag\\
F_2(s)&=&\dfrac{3(s-1)\zeta(s)}{2^s\zeta(2)(s-2)}
\end{eqnarray*}
and thus 
\begin{eqnarray}\label{deco1}
J_{3,2}(x)=K_{2}(x)-K_1(x)
\end{eqnarray}
where 
\begin{eqnarray*}
K_1(x)&=&\dfrac{1}{2\pi i}\int_{\mathcal{C}}F_1(s)\dfrac{x^s}{s(s+1)}ds\notag\\
K_2(x)&=&\dfrac{1}{2\pi i}\int_{\mathcal{C}}F_2(s)\dfrac{x^s}{s(s+1)}ds
\end{eqnarray*}
It is clear that the integrand in $K_1(x)$ has a simple pole at $s=1$ and that in $K_2(x)$ has a simple pole at $s=2$. We now estimate $K_1(x)$ and $K_2(x)$ separately. 
\subsubsection{\underline{Estimation of $K_1(x)$}}
To estimate $K_1(x)$, we first shift the contour $\mathcal{C}$ to the left of the point $s=2$. Let $\sigma'$ be as in (\ref{parre}). Since the integrand in $K_1(x)$  is analytic inside and on the rectangle $\mathcal{R}_1$, we can easily see using Cauchy's residue theorem and Theorem \ref{zetaest} that 
\begin{eqnarray}
\dfrac{1}{2\pi i}\int_{\gamma_1\cup\gamma_2}F_1(s)\dfrac{x^s}{s(s+1)}ds\ll \dfrac{x^\sigma\log T'}{T'^2}\rightarrow 0,\;\;T'\rightarrow\infty
\end{eqnarray}
and thus
\begin{eqnarray}
K_1(x)=\dfrac{1}{2\pi i}\int_{\mathcal{C}'}F_1(s)\dfrac{x^s}{s(s+1)}ds.
\end{eqnarray}
Now we truncate the contour $\mathcal{C}'$ at a height $1\leq T\leq x$ as before and deform this to the rectangle $\mathcal{R}'$. Since $s=1$ is a simple pole of $F_1(s)x^s/s(s+1)$, it follows from Cauchy's residue theorem that 
\begin{eqnarray}\label{k1res}
\dfrac{1}{2\pi i}\int_{\mathcal{R}'}F_1(s)\dfrac{x^s}{s(s+1)}ds=\dfrac{3x}{4\zeta(2)}.
\end{eqnarray}
Using Theorem \ref{zetaest}, it follows easily that
\begin{eqnarray}\label{errk0}
\dfrac{1}{2\pi i}\int_{C_T'}F_1(s)\dfrac{x^s}{s(s+1)}ds \ll \dfrac{x\log T}{T}
\end{eqnarray}
and 
\begin{eqnarray}\label{errk1}
\dfrac{1}{2\pi i}\int_{C_2'\cup C_4'}F_1(s)\dfrac{x^s}{s(s+1)}ds&\ll &\dfrac{x\log T}{T^2}(\sigma'-\sigma_1)\ll\dfrac{x}{T^2}\notag\\
\dfrac{1}{2\pi i}\int_{C_3'}F_1(s)\dfrac{x^s}{s(s+1)}ds&\ll & \dfrac{x^{\sigma_1}\log T}{T}+\dfrac{x^{\sigma_1}}{1-\sigma_1}\notag\\
&\ll & x^{\sigma_1}\log T.  
\end{eqnarray}  
\subsubsection{\underline{Estimation of $K_2(x)$}}
In order to estimate $K_2(x)$, we truncate the contour $\mathcal{C}$ at a height $1\leq W\leq x$ (to be chosen later). Let $C_W=[\sigma\pm i\infty, \sigma\pm iW]$. Choose 
\begin{eqnarray}
\sigma &=& 2+\dfrac{1}{\log x}\notag\\
\sigma_2&= &1+\dfrac{1}{\log x}.
\end{eqnarray}
Now, deform this truncated contour to a rectangle $\mathcal{R}_2$ as shown below. \newline
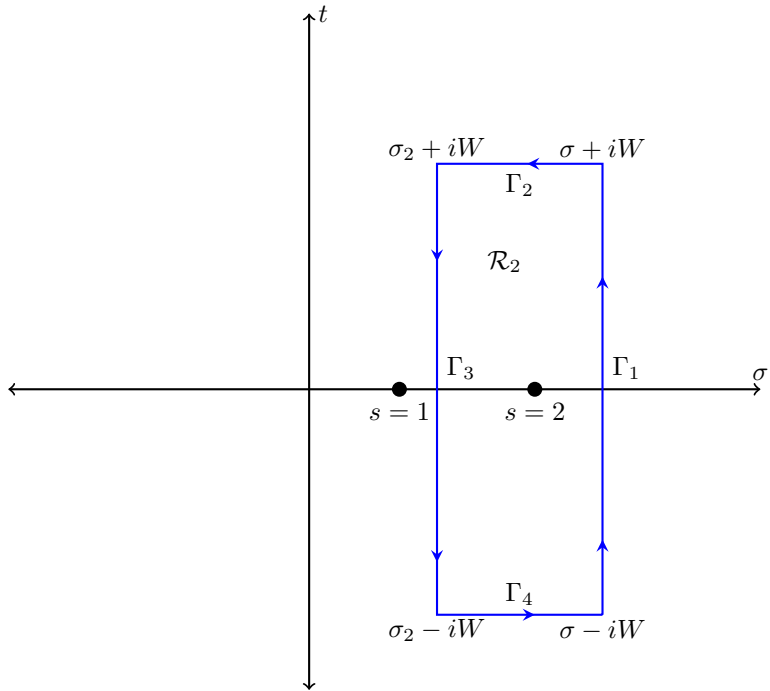
\begin{tikzpicture}
[decoration={markings,
    mark=at position 1cm   with {\arrow[line width=1pt]{stealth}},
    mark=at position 4.5cm with {\arrow[line width=1pt]{stealth}},
    mark=at position 7cm   with {\arrow[line width=1pt]{stealth}},
    mark=at position 9.5cm with {\arrow[line width=1pt]{stealth}},
    mark=at position 13.5cm with {\arrow[line width=1pt]{stealth}},
    mark=at position 15.5cm with {\arrow[line width=1pt]{stealth}}
  }]
\draw[thick, <->] (-4,0) -- (6,0) coordinate (xaxis);

  \draw[thick, <->] (0,-4) -- (0,5) coordinate (yaxis);

  \node[above] at (xaxis) {$\sigma$};

  \node[right]  at (yaxis) {$t$};
\path[draw,blue, line width=0.8pt, postaction=decorate] (3.9,-3)
    -- node[midway, above right, black] {$\Gamma_1$} (3.9,3)
    -- node[midway, below, black] {$\Gamma_2$}(1.7,3)
    -- node[midway, above right, black] {$\Gamma_3$} (1.7,-3) 
    -- node[midway, above, black] {$\Gamma_4$}(3.9,-3);
\draw (1.2,0) node [circle, inner sep=2pt, fill]{};
\draw (3,0) node [circle, inner sep=2pt, fill]{};
\draw (1.2,-0.3) node{$s=1$};
\draw (3,-0.3) node{$s=2$};
\draw (3.9,3.2) node{$\tiny{\sigma+iW}$};
\draw (1.7,3.2) node{$\tiny{\sigma_2+iW}$};
\draw (1.7,-3.2) node{$\tiny{\sigma_2-iW}$};
\draw (3.9,-3.2) node{$\tiny{\sigma-iW}$};
\draw (2.6,1.7) node{$\mathcal{R}_2$};
\end{tikzpicture}\newline
Since $s=2$ is a simple pole for the integrand in $K_{2}(x)$, using Cauchy's residue theorem we find that 
\begin{eqnarray}\label{Jest3}
\displaystyle\text{Res}_{s=2}\left(\dfrac{x^sF_2(s)}{s(s+1)}\right)=\dfrac{x^2}{8}
\end{eqnarray}
The contribution of the integral $K_{2}(x)$ on the contour $C_W$ can be obtained by using 
\begin{eqnarray*}
|\zeta(s)|\leq \zeta(2),\;\;-\infty<t<\infty
\end{eqnarray*}
and we get
\begin{eqnarray}\label{j321est}
\dfrac{1}{2\pi i}\int_{C_W}F_2(s)\dfrac{x^s}{s(s+1)}&\ll &\dfrac{x^2}{W}.
\end{eqnarray}
To obtain the contribution of $K_2(x)$ on $\Gamma_2$ and $\Gamma_4$ , we use Theorem \ref{zetaest}. We thus get
\begin{eqnarray}\label{j32est}
\dfrac{1}{2\pi i}\int_{\Gamma_2\cup \Gamma_4}F_2(s)\dfrac{x^s}{s(s+1)}&\ll & \dfrac{x^2\log W}{W^2}.
\end{eqnarray}  
To obtain the contribution of the integral on $\Gamma_3$, we again use Theorem \ref{zetaest}, but this time we use the estimate that $\log (\zeta(s)(s-1))\ll 1$ when $|t|$ is small. We thus obtain
\begin{eqnarray}\label{j322est}
\dfrac{1}{2\pi i}\int_{\Gamma_3}F_2(s)\dfrac{x^s}{s(s+1)}&\ll & x^{\sigma_2}\log W\int_{-W}^W\dfrac{dt}{(1+|t|)^2}+x^{\sigma_2}\int_{-1}^1\dfrac{dt}{|\sigma_2-2+it||\sigma_2+it||\sigma_2+1+it|} \notag\\
&\ll & \dfrac{x^{\sigma_2}\log W}{W}+x^{\sigma_2}\ll x.
\end{eqnarray}
Thus it follows from (\ref{deco1}), (\ref{errk0}), (\ref{errk1}), (\ref{k1res}), (\ref{Jest3}), (\ref{j321est}), (\ref{j32est}) and (\ref{j322est}) that
\begin{eqnarray}\label{j32fin}
J_{3,2}(x)&=&\dfrac{x^2}{8}-\dfrac{3x}{4\zeta(2)}+O\left(\dfrac{x\log T}{T}\right)+O\left(x^{\sigma_1}\log T\right)+O\left(\dfrac{x^2}{W}\right)
\end{eqnarray}
and from (\ref{Jfin}), (\ref{Jest2}) and (\ref{j32fin}) we obtain
\begin{eqnarray}\label{j3fin}
J_3(x)=O\left(\dfrac{x\log T}{T}\right)+O\left(x^{\sigma_1}\log T\right)+O\left(\dfrac{x^2}{W}\right)+O(x).
\end{eqnarray}
\section{Proof of Theorem \ref{main}}
From (\ref{Tphi}), (\ref{resf}), (\ref{resfd1}), (\ref{e2}), (\ref{e7}), (\ref{e4}), (\ref{e5}), (\ref{e6p}), (\ref{e6p}) and (\ref{j3fin}), we get
\begin{eqnarray}\label{Tphifin}
T_\varphi^a(x)=\dfrac{x\log x}{2\zeta(2)}+E_{T,W}(x)
\end{eqnarray}
where 
\begin{eqnarray}
E_{T,W}(x)&=&O\left(\dfrac{x\log T}{T}\right)+O\left(x^{\sigma_1}(\log T)^2\right)+O\left(\dfrac{x^2}{W}\right)+O(x).
\end{eqnarray}
We now choose 
\begin{eqnarray*}
T&=&\exp\left({\sqrt{c\log x}}\right)\notag\\
W&=&x
\end{eqnarray*}
to find that 
\begin{eqnarray}\label{err23}
E_{T,W}(x)=O(x).
\end{eqnarray}
Thus from (\ref{Sphi}), (\ref{Tphifin}) and (\ref{err23}) we get
\begin{eqnarray}
\int_{1}^xS_\varphi(t)\;dt=\dfrac{x^2\log x}{2\zeta(2)}+O(x^2).
\end{eqnarray} 
\section{Proof of Corollary \ref{cor1}}
Let for $t\geq 2$, there exists an $\alpha\in\mathbb{R}$ such that 
\begin{eqnarray}\label{asy}
S_\varphi(t)=\alpha t\log t(1+o(1)).
\end{eqnarray}
Then integrating both sides of (\ref{asy}) in $[1,x]$, we get 
\begin{eqnarray}\label{asy1}
\int_{1}^xS_\varphi(t)\;dt=\dfrac{\alpha x^2\log x(1+o(1))}{2}
\end{eqnarray}
Comparing (\ref{asy1}) and Theorem \ref{main}, we immediately see that
\begin{eqnarray*}
\alpha=\dfrac{1}{\zeta(2)}+o(1),\;\;x\rightarrow\infty
\end{eqnarray*} 
and thus the result follows.
\section{Proof of Corollary \ref{cor2}}
Let 
\begin{eqnarray}\label{req0}
S_\varphi(x)=S^{*}_\varphi(x)+E_\varphi(x)
\end{eqnarray}
where $E_\varphi(x)=o(x\log x)$. First, choose any $\beta_1>1$. Then from Theorem \ref{main} we have
\begin{eqnarray}\label{req1}
\int_{x}^{\beta_1 x}S_\varphi(t)\;dt&=&\int_{1}^{\beta_1 x}S_\varphi(t)\;dt-\int_{1}^{x}S_\varphi(t)\;dt\notag\\
&=&\dfrac{(\beta_1 x)^2\log (\beta_1 x)}{2\zeta(2)}-\dfrac{x^2\log (x)}{2\zeta(2)}+O(x^2)\notag\\
&=&\dfrac{x^2\log x}{2\zeta(2)}\left(\beta_1^2-1\right)+O(x^2)
\end{eqnarray}
Since $S^{*}_\varphi(t)$ is monotonically non-decreasing, from (\ref{req0}) we have
\begin{eqnarray}\label{req2}
\int_{x}^{\beta_1 x}S_\varphi(t)\;dt&=&\int_{x}^{\beta_1 x}S^{*}_\varphi(t)\;dt+o(x^2\log x)\notag\\
&\geq &xS^{*}_{\varphi}(x)(\beta_1-1)+o(x^2\log x).
\end{eqnarray}
Hence, (\ref{req1}) and (\ref{req2}) yield
\begin{eqnarray}\label{req3}
S^{*}_{\varphi}(x)\leq \dfrac{x\log x}{2\zeta(2)}\cdot\left(\dfrac{\beta_1^2-1}{\beta_1-1}\right)(1+o(1)).
\end{eqnarray}
Keeping $\beta_1$ fixed and letting $x\rightarrow \infty$ on both sides of (\ref{req3}) we get
\begin{eqnarray}\label{req4}
\limsup_{x\rightarrow\infty}\dfrac{S^{*}_\varphi(x)}{x\log x}\leq \dfrac{1}{2\zeta(2)}\left(\dfrac{\beta_1^2-1}{\beta_1-1}\right).
\end{eqnarray}
Now let $\beta_1\rightarrow 1^+$ in (\ref{req4}) to obtain
\begin{eqnarray}
\limsup_{x\rightarrow\infty}\dfrac{S^{*}_\varphi(x)}{x\log x}\leq \dfrac{1}{\zeta(2)}.
\end{eqnarray}
Now consider any $\beta_2$ with $0<\beta_2<1$ and consider the difference $S_\varphi(x)-S_\varphi(\beta_2 x)$. An argument similar to the above shows that 
\begin{eqnarray*}
\liminf_{x\rightarrow\infty}\dfrac{S^{*}_{\varphi}(x)}{x\log x}\geq \dfrac{1}{2\zeta(2)}\left(\dfrac{1-\beta_2^2}{1-\beta_2}\right).
\end{eqnarray*}
As $\beta_2\rightarrow 1^-$, the right hand side above tends to $1/\zeta(2)$. This together with (\ref{req4}) and the fact that $E_\varphi(x)=o(x\log x)$ proves the corollary. 
\section{Proof of Corollary $\ref{cor3}$}
From Theorem \ref{main} we have
\begin{eqnarray}\label{ch2}
\int_{x}^{x+h}S_\varphi(t)\;dt=\dfrac{hx\log x}{\zeta(2)}+O(h^2\log x)+O(x^2).
\end{eqnarray}
Again from (\ref{ch2}) we obtain
\begin{eqnarray}\label{ch6}
hS_\varphi(x)=\dfrac{hx\log x}{\zeta(2)}+O(h^2\log x)+O(x^2)-L
\end{eqnarray}
where 
\begin{eqnarray}\label{ch3}
L=\int_{x}^{x+h}(S_\varphi(t)-S_{\varphi}(x))\;dt.
\end{eqnarray}
Since $L$ satisfies (\ref{coreq1}), from (\ref{ch6}) and (\ref{ch3}) we have
\begin{eqnarray}
S_\varphi(x)=\dfrac{x\log x}{\zeta(2)}+O(h\log x)+O\left(\dfrac{x^2}{h}\right).
\end{eqnarray}
By choosing $h=x(\log x)^{-1/2}$, we see that 
\begin{eqnarray}
S_\varphi(x)=\dfrac{x\log x}{\zeta(2)}+O(x\sqrt{\log x}).
\end{eqnarray}
\section{Proof of Corollary $\ref{cor4}$}
It is clear from (\ref{Tphifin}) that $T^a_\varphi(x)$ satisfies
\begin{eqnarray}\label{eqeq1}
\lim_{x\rightarrow\infty}\dfrac{T^a_{\varphi}(x)}{x\log x}=\dfrac{1}{2\zeta(2)}.
\end{eqnarray}
From (\ref{Sphi}), we have for any $h>0$,
\begin{eqnarray}\label{eqeq4}
S_\varphi(x)&=&\sum_{n\leq x}\Phi(n)+O(x)\notag\\
&=&\dfrac{x+h}{h}\sum_{n\leq x+h}\Phi(n)\left(1-\dfrac{n}{x+h}\right)-\dfrac{x}{h}\sum_{n\leq x}\Phi(n)\left(1-\dfrac{n}{x}\right)\notag\\&&-\dfrac{1}{h}\sum_{x<n<x+h}\Phi(n)(x+h-n)+O(x)\notag\\
&=& \dfrac{x+h}{h}\cdot T^a_{\varphi}(x+h)-\dfrac{x}{h}\cdot T^a_{\varphi}(x)-U_{\varphi}(x)+O(x)
\end{eqnarray}
where $U_{\varphi}(x)=h^{-1}\sum_{x<n<x+h}\Phi(n)(x+h-n)$. We take $h=\varepsilon x$ for some $\varepsilon>0$. From (\ref{eqeq1}), we see that
\begin{eqnarray}\label{eqeq2}
\lim_{x\rightarrow\infty}\dfrac{1}{x\log x}\left(\dfrac{x+h}{h}\cdot T^{a}_{\varphi}(x+h)\right)&=&\dfrac{1}{2\zeta(2)}\cdot \dfrac{(1+\varepsilon)^2}{\varepsilon},\notag\\
\lim_{x\rightarrow\infty}\dfrac{1}{x\log x}\left(\dfrac{x}{h}\cdot T^{a}_{\varphi}(x)\right)&=&\dfrac{1}{2\zeta(2)}\cdot \dfrac{1}{\varepsilon}.
\end{eqnarray}
From (\ref{bou}), we see that
\begin{eqnarray*}
U_{\varphi}(x)\geq -C\sum_{x<n<x+h}\log n\geq -Ch\log x=-C\varepsilon x\log x
\end{eqnarray*}
and hence 
\begin{eqnarray}\label{eqeq3}
\liminf_{x\rightarrow\infty}\dfrac{U_{\varphi}(x)}{x\log x}\geq -C\varepsilon.
\end{eqnarray}
Combining (\ref{eqeq4}), (\ref{eqeq2}) and (\ref{eqeq3}) we get
\begin{eqnarray}
\limsup_{x\rightarrow\infty}\dfrac{S_\varphi(x)}{x\log x}\leq \dfrac{1}{\zeta(2)}+\varepsilon\left(C+1/2\zeta(2)\right).
\end{eqnarray}
Since $\varepsilon$ can take arbitrarily small values, it follows that
\begin{eqnarray}
\limsup_{x\rightarrow\infty}\dfrac{S_\varphi(x)}{x\log x}\leq \dfrac{1}{\zeta(2)}.
\end{eqnarray} 
To obtain a corresponding lower bound, we note that
\begin{eqnarray}
S_\varphi(x)&=&\sum_{n\leq x}\Phi(n)+O(x)\notag\\
&=&\dfrac{x}{h}\sum_{n\leq x}\Phi(n)\left(1-\dfrac{n}{x}\right)-\dfrac{x-h}{h}\sum_{n\leq x-h}\Phi(n)\left(1-\dfrac{n}{x-h}\right)\notag\\&&+\dfrac{1}{h}\sum_{x-h<n<x}\Phi(n)(n+h-x)+O(x)\notag\\
&=& \dfrac{x}{h}\cdot T^a_{\varphi}(x)-\dfrac{x-h}{h}\cdot T^a_{\varphi}(x-h)+V_{\varphi}(x)+O(x)
\end{eqnarray}
where $V_{\varphi}(x)=h^{-1}\sum_{x-h<n<x}\Phi(n)(n+h-x)$. Arguing as we did before, we obtain
\begin{eqnarray}
\liminf_{x\rightarrow\infty}\dfrac{S_\varphi(x)}{x\log x}\geq \dfrac{1}{\zeta(2)}-\varepsilon(1/2\zeta(2)+C/2)
\end{eqnarray}
so that
\begin{eqnarray}
\liminf_{x\rightarrow\infty}\dfrac{S_\varphi(x)}{x\log x}\geq \dfrac{1}{\zeta(2)}
\end{eqnarray}
and thus the result follows.
\section{Conclusions}
We note that our method does not make any use of the rate of growth of the totient function $\varphi(n)$. It mainly uses the representation (\ref{Dphi1}) for the Dirichlet series generating function of $S_\varphi(x)$ in terms of the zeta functions, and as such our method may be used to treat sums with more general functions involved, provided the corresponding Dirichlet series have similar representations. 

Secondly, our repesentation of the Dirichlet series $D_\varphi(s)$ in terms of the zeta functions relies on Lemma \ref{Apo}. Using this representation, it is really difficult to apply Perron's formula directly. One of the difficulties is when we truncate the contour at a height $T$, the truncation error we obtain is through analysis of coefficients of the Dirichlet series and in our case, the integral $J_3(x)$ poses some problems along with $I_3(s)$ which does not have any explicit form. Thus we have to consider the weighted sum where we do not require any such analysis. However, if one can obtain a better representation for $D_\varphi(s)$, it might be possibe to apply Perron's formula directly to obtain an asymptotic formula for $S_\varphi(x)$. 

Finally one can consider, for example, the following Riesz-type weighted sum:
\begin{eqnarray}
T^r_{\varphi}(x):=\sum_{n\leq x}\Phi(n)\log (x/n)=\int_{1}^x\dfrac{S_\varphi(t)}{t}\;dt+O(x).
\end{eqnarray}         
It can be shown via a similar analysis that
\begin{eqnarray}
\int_{1}^x\dfrac{S_\varphi(t)}{t}\;dt\sim\dfrac{x\log x}{\zeta(2)},\;x\rightarrow\infty.
\end{eqnarray}
We leave the details of proof to the interested reader. 
\section{Acknowledgement}
The author is indebted to Prof. K. Alladi for his support, encouragement and stimulating discussions. He greatly acknowledges his valuable suggestions on an earlier draft of the manuscript which fixed a few errors and typos besides improving the exposition. The suggestion to consider a Ces\`{a}ro-type weighted sum is due to Prof. K. Alladi.     


\begin{thebibliography}{9}
\bibitem{Apo} 
Tom Apostol. 
\textit{Introduction to Analytic Number Theory}. 
Springer International Student Edition, 1998.
\bibitem{Igor} 
O. Bordell\'{e}s, L. Dai, R. Heyman, H. Pan, I.R. Shparlinski. 
\textit{On a sum involving the Euler totient function}. 
arXiv: 1808.00188v3, 15 October 2018.
\bibitem{Chern} 
Shane Chern.
\textit{Notes on sums involving the Euler function}. 
Bull. of Aus. Math. Soc. (to appear), 11 December 2018.
\bibitem{Mont} 
Hugh L. Montgomery, Robert C. Vaughan .
\textit{Multiplicative Number Theory I-Classical Theory}. 
Cambridge University Press, 2006.
\bibitem{Rada} 
Hans Rademacher.
\textit{Topics in analytic number theory}. 
Springer-Verlag, New York Heildelberg Berlin, 1973.
\bibitem{Ten} 
G. Tenenbaum.
\textit{Introduction to analytic and probabilisic number theory}. 
Grad. Studies. Math, vol. 163, Amer. Math. Soc., 2015.
\bibitem{Wu} 
J. Wu. 
\textit{On a sum involving the Euler totient function}. 
Hal. 01884018v2, 3 October 2018.
\end{thebibliography}
\end{document}